\documentclass[a4paper,12pt]{article} 
\usepackage{natbib}
\usepackage[utf8]{inputenc}
\usepackage[T1]{fontenc}
\usepackage{lmodern}
\usepackage{amsmath} 
\usepackage{amsthm}
\usepackage{amssymb}
\usepackage{mathtools}
\usepackage{overpic}
\usepackage{fancyhdr}
\usepackage{graphicx} 
\usepackage[left=1in,top=1in,bottom=1.5in,right=1in]{geometry}

\renewcommand{\subsectionmark}[1]{} 
\newcommand{\eps}{\varepsilon}

\newcommand{\T}{\mathbb{T}}

\newcommand*{\C}{{\mathbb{C}}}     
\newcommand*{\R}{{\mathbb{R}}}     
     
\newcommand*{\N}{{\mathbb{N}}}

\newcommand*{\Lin}{{\mathcal{L}}}   
\newcommand*{\Dom}{{\mathcal{D}}}   
\newcommand{\ran}{{\mathcal{R}}}   
\renewcommand{\ker}{{\mathcal{N}}}

\newcommand*{\abs} [1]{\lvert#1\rvert}
\newcommand*{\norm}[1]{\lVert#1\rVert}
\newcommand*{\set} [1]{\{#1\}}
\newcommand*{\setm}[2]{\{\,#1\mid#2\,\}}   
\newcommand*{\iprod}[2]{\langle#1,#2\rangle}    
\newcommand*{\Setm}[2]{\left\{\,#1\,\middle|\,#2\,\right\}}

\newcommand*{\Abs}[2][default]{\ifthenelse{\equal{#1}{default}}{\left\lvert#2\right\rvert}{\ldelim{#1}{\lvert}#2\rdelim{#1}{\rvert}}}
\newcommand*{\Norm}[2][default]{\ifthenelse{\equal{#1}{default}}{\left\lVert#2\right\rVert}{\ldelim{#1}{\lVert}#2\rdelim{#1}{\rVert}}}

\newcommand*{\Iprod}[3][default]{\ifthenelse{\equal{#1}{default}}{\left\langle#2,#3\right\rangle}{\ldelim{#1}{\langle}#2,#3\rdelim{#1}{\rangle}}}
\newcommand*{\Dualpair}[3][default]{\ifthenelse{\equal{#1}{default}}{\left\langle#2,#3\right\rangle}{\ldelim{#1}{\langle}#2,#3\rdelim{#1}{\rangle}}}

\newcommand*{\List}[2][1]{\set{#1,\ldots,#2}}

\newcommand{\eq}[1]{\begin{align*}#1\end{align*}}
\newcommand{\eqn}[1]{\begin{align}#1\end{align}}

\newcommand{\Igw}{\List{N}} 

\newcommand{\Y}{\C^p}
\newcommand{\D}{\mathbb{D}}
\newcommand{\gs}{\sigma}
\newcommand{\ga}{\alpha}
\newcommand{\gb}{\beta}
\renewcommand{\gg}{\gamma}
\newcommand{\gd}{\delta}
\newcommand{\gl}{\lambda}

\newcommand{\ieq}[1]{$#1$}

\newcommand{\inv}{^{-1}}
\newcommand*{\floor}[1]{\lfloor#1\rfloor}

\newcommand{\Asg}[1][A]{({#1}^n)_{n\in\N}}

\newcommand{\ep}[1][k]{e^{i\varphi_{#1}}}
\newcommand{\epc}[1][k]{e^{-i\varphi_{#1}}}

\newcommand{\BA}{\Lambda}

\newtheorem{theorem}{Theorem}
\newtheorem{lemma}[theorem]{Lemma}
\newtheorem{corollary}[theorem]{Corollary}

\newtheorem{assumption}[theorem]{Assumption}

\theoremstyle{definition}

\bibpunct{[}{]}{,}{n}{}{,}

\title{Robustness of Strong Stability of Discrete Semigroups}
\author{Lassi Paunonen\thanks{Tampere University of Technology, PO.Box 553, 33101 Tampere, Finland, \texttt{lassi.paunonen@tut.fi}}}
\date{~}

\begin{document}

\maketitle
\vspace{-8ex}

\thispagestyle{plain}

\begin{abstract}
In this paper we study the robustness of strong stability of a discrete semigroup on a Hilbert space under bounded finite rank perturbations. As the main result we characterize classes of perturbations preserving the strong stability of the semigroup. 
\end{abstract}

\section{Introduction}

Due to the high level of generality and the many forms of strong stability, finding conditions for preservation of strong stability of a semigroup under perturbations of its generator is a challenging research problem.  However, recent advances in the theory of nonuniform stability of semigroups~\cite{BatEng06,BatDuy08,BorTom10,BatChi14} have made it possible to study robustness of stability for semigroups that are not exponentially stable~\cite{Pau12,Pau13c-a}. While general strongly stable semigroups may have no intrinsic robustness properties, the theory of nonuniform stability of semigroups opens doors for research on robustness properties for many important subclasses of strongly stable semigroups.

In this short paper we consider the preservation of strong stability of discrete semigroups $\Asg$ with $A\in \Lin(X)$ under additive finite rank perturbations $A+BC$ with $B\in \Lin(\Y,X)$ and $C\in \Lin(X,\Y)$. In particular, we assume that the unperturbed semigroup $\Asg$ is strongly stable in such a way that $A$ has a finite number of spectral points on the unit circle $\T$, and the growth of its resolvent operator is polynomially bounded near these points.  

The main result of this paper is a discrete analogue of the set of conditions for preservation of strong stability of strongly continuous semigroups presented in~\cite{Pau13c-a}. The techniques employed here are similar to those used in~\cite{Pau13c-a}, but in many situations the proofs can be greatly simplified due to the fact that the operator $A$ is bounded. The discrete proofs also require several modifications, mainly in estimating the behaviour of the resolvent operator near the unit disk $\D$. To the author's knowledge, the preservation of strong stability of discrete semigroups has not been studied previously in the literature. Moreover, the resolvent estimates presented in this paper generalize the results found in the literature by allowing $A$ to have multiple spectral points on $\T$.

Assumption~\ref{ass:Astandass} below states the standing assumptions on the semigroup $\Asg$ and on the perturbations. The strong stability of $\Asg$ implies that $\gs_p(A)\cap \T=\varnothing$. Since $X$ is a Hilbert space, Theorem 2.9 and Corollary 2.11 in~\cite{Eis10book} imply that for all $\varphi\in[0,2\pi]$ 
\eq{
X = \ker(A-\ep[]) \oplus \overline{\ran(A-\ep[])} = \overline{\ran(A-\ep[])}.
}
Therefore, all spectral points of $A$ on the unit circle belong to $\gs_c(A)$.

\begin{assumption} 
  \label{ass:Astandass} 
  Let $X$ be a Hilbert space. Assume that the operators $A \in \Lin(X)$, $B\in \Lin(\Y,X)$, and $C \in \Lin(X,\Y)$ satisfy the following for some $\ga\geq 1$, and $\gb,\gg\geq 0$.
  \begin{itemize}
    \item[\textup{1.}]
      The discrete semigroup $\Asg$ is strongly stable, $\gs(A)\cap \T = \set{\ep}_{k=1}^N$ for some $N\in\N$ and $d_A=\min_{k\neq l}\abs{\varphi_k-\varphi_l}>0$.
      Moreover, there exist constants $M_A\geq 1$ and  $0<\eps_A\leq \min \set{\pi/8,d_A/3}$ such that
      \eqn{
        \label{eq:Aresolventgrowthorder}
	\sup_{0<\abs{\varphi-\varphi_k}\leq \eps_A} \abs{\varphi-\varphi_k}^\ga \norm{R(\ep[],A)}\leq M_A,
      }
      for all $k\in\Igw$
      and $\norm{R(\ep[],A)}\leq M_A$ whenever $\abs{\varphi-\varphi_k}>\eps_A$ for all $k$. 
    \item[\textup{2.}] 
      $\ran(B)\subset \ran( (1- \epc A)^\gb)$ and $\ran(C^\ast)\subset \ran( (1-\ep A^\ast)^\gg)$ for every $k\in\Igw$.
  \end{itemize}
\end{assumption}

The second part of Assumption~\ref{ass:Astandass} together with the Closed Graph Theorem implies $(1-\epc A)^{-\gb} B\in \Lin(\Y,X)$ and $(1-\ep A^\ast)^{-\gg} C^\ast \in \Lin(X,\Y)$.

The following theorem presenting conditions for preservation of the stability of the semigroup $\Asg$ is the main result of this paper. 

\begin{theorem}
  \label{thm:stabpert}
  Let Assumption~\textup{\ref{ass:Astandass}} be satisfied with $\gb+\gg\geq \ga$. There exists $\gd>0$ such that if
  \eq{
  \norm{(1-e^{-i\varphi_k} A)^{-\gb} B}<\gd, \quad \mbox{and} \quad  \norm{(1-e^{i\varphi_k} A^\ast)^{-\gg}C^\ast}<\gd
  }
  for all $k\in \Igw$, then the discrete semigroup $\Asg[(A+BC)]$ is strongly stable. 
  Moreover, we then have $\gs(A+BC)\cap \T = \gs_c(A+BC)\cap \T = \set{\ep}_{k=1}^N$, and for all $k$
  \eq{
  \sup_{0<\abs{\varphi-\varphi_k}\leq\eps_A}\abs{\varphi-\varphi_k}^\ga\norm{R(\ep[],A+BC)} <\infty.
  }
\end{theorem}

We begin the paper by studying the behaviour of the resolvent operator $R(\gl,A)$ near the unit disk $\D$ in Section~\ref{sec:resest}. These results are required in the proof of Theorem~\ref{thm:stabpert}, which is presented subsequently in Section~\ref{sec:strongstabpres}.

If~$X$ and~$Y$ are Banach spaces and~$A:X\rightarrow Y$ is a linear operator, we denote by~$\Dom(A)$, $\ran(A)$, and $\ker(A)$ the domain, the range, and the kernel of~$A$, respectively. 
The space of bounded linear operators from~$X$ to~$Y$ is denoted by~$\Lin(X,Y)$. 
If \mbox{$A:\Dom(A)\subset X\rightarrow X$,} then~$\gs(A)$, $\gs_p(A)$, $\gs_c(A)$ and~$\rho(A)$ denote the spectrum, the point spectrum, the continuous spectrum and the \mbox{resolvent} set of~$A$, respectively. For~$\gl\in\rho(A)$ the resolvent operator is given by \mbox{$R(\gl,A)=(\gl -A)^{-1}$}. 
The inner product on a Hilbert space is denoted by $\iprod{\cdot}{\cdot}$.  
We denote $\T = \setm{z\in\C}{\abs{z}=1}$, $\D = \setm{z\in\C}{\abs{z}<1}$, $\overline{\D} = \setm{z\in\C}{\abs{z}\leq 1}$,

\section{Resolvent Estimates}
\label{sec:resest}

In this section we study the behaviour of the resolvent operator $R(\gl,A)$ near the unit disk $\D$. In particular, the proof of Theorem~\ref{thm:stabpert} is based on the property that the polynomial growth of the resolvent operator near the points $\ep$ can be cancelled by a suitable operator. The general form of the resolvent estimates follows the recent results for strongly continuous semigroups that have appeared in~\cite{BorTom10,LatShv01,BatChi14}, and the results in this section can be seen as straightforward discrete reformulations of corresponding results in the previous references. The main difference compared to the previous references is that we allow the operator $A$ to have multiple spectral points on the unit circle $\T$.

Define $\BA_k = 1-\epc A$ for $k\in\Igw$. The operators $\BA_k$ and $\BA_l$ commute for every $k,l\in \Igw$, we have $\BA_k^\ast = 1-\ep A^\ast$, and the families $(\BA_k)_{k=1}^N$ and $(\BA_k^\ast)_{k=1}^N$ are uniformly sectorial. Indeed, since the operator $A$ is power bounded, the strong Kreiss resolvent condition~\cite{Eis10book} implies 
\ieq{
\norm{R(\gl,\epc A)} \leq M/(\abs{\gl}-1)
}
for all $\gl\in\C\setminus \overline{\D}$,
where $M=\sup_{n\in\N} \norm{A^n} =\sup_{n\in\N} \norm{(\epc A)^n} $.
This implies that for every $\gl>0$ we have
\eq{
\norm{\gl (\gl+1-\epc A)\inv}
\leq \gl\frac{M}{\abs{\gl+1}-1}
=M.
}
Since the bound is independent of $\varphi_k \in [0,2\pi]$, by~\cite[Prop. 2.1.1]{Haa06book} the family $(\BA_k)_{k=1}^N$ is uniformly sectorial. Since $\gs_p(A)\cap \T=\varnothing$, the operators $\BA_k$ are injective and have sectorial inverses $\BA_k\inv: \ran(\BA_k)\subset X\rightarrow X$. The same conclusions are true for the operators $\BA_k^\ast = 1-\ep A^\ast$. The fractional powers $\BA_k^\gb$ and $(\BA_k^\ast)^\gg$ are therefore defined for all $\gb,\gg\in\R$.

Consider regions $\Omega_k\subset \C\setminus\D$ defined as
(see Figure~\ref{fig:Omegak}). 
\eq{
\Omega_k = \Setm{\gl\in \C}{\abs{\gl}\geq 1, ~0< \abs{\gl-\ep} \leq r_A},
}
where $r_A = \abs{1-e^{i\eps_A}}$. We have $0<r_A\leq 1$ and $\abs{e^{i\varphi_k}-e^{i(\varphi_k\pm\eps_A)}}=r_A$ for all $k$.

\begin{figure}[ht]
  \begin{center}
    \includegraphics[width=0.4\linewidth]{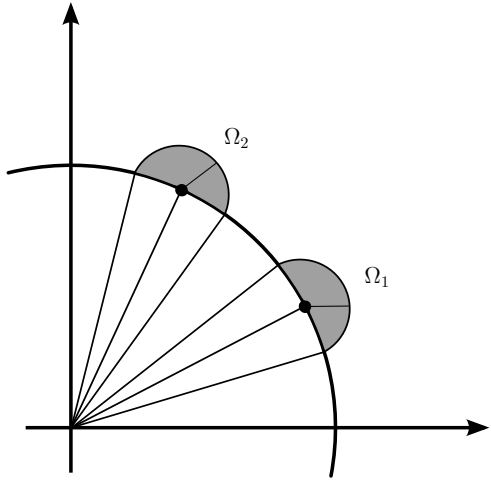}
  \end{center}
  \caption{The domains $\Omega_k$.}
  \label{fig:Omegak}
\end{figure}

The following is the main resolvent estimate required in the proof of Theorem~\ref{thm:stabpert}.

\begin{theorem} 
  \label{thm:RBAbdd}
  If Assumption~\textup{\ref{ass:Astandass}} is satisfied, there exists $M_1\geq 1$ such that 
  \eq{
  \sup_{\gl\in\Omega_k}  \norm{R(\gl,A)\BA_k^\ga}\leq M_1
  }
  for all $k\in\Igw$.
\end{theorem}

The proof of the theorem is based on the following two lemmas. The Moment Inequality in Lemma~\ref{lem:Momentineq} is an essential tool used frequently throughout the rest of the paper.

\begin{lemma}
  \label{lem:Momentineq}
Let $0<\tilde{\theta}<\theta$. There exists $M_{\tilde{\theta}}\geq 1$
  such that for all $k\in\Igw$ 
  \eq{
  \norm{\BA_k^{\tilde{\theta}}x} 
  &\leq M_{\tilde{\theta}} \norm{x}^{1-\tilde{\theta}/\theta}\norm{\BA_k^{\theta}x}^{\tilde{\theta}/\theta} \qquad \forall x\in X.
  }
  If $Y$ is a Banach space and $R\in \Lin(Y,X)$, then
  \eq{
  \norm{\BA_k^{\tilde{\theta}}R} 
  &\leq M_{\tilde{\theta}} \norm{R}^{1-\tilde{\theta}/\theta}\norm{\BA_k^{\theta}R}^{\tilde{\theta}/\theta}
  }
  for all $k$. The corresponding results are valid for $(\BA_k^\ast)_{k}$. 
\end{lemma}

\begin{proof}
  For a fixed $k$ the properties follow from~\cite[Prop. 6.6.4]{Haa06book}. However, by~\cite[Prop. 2.6.11]{Haa06book} and the uniform sectoriality of the operator family $(\BA_k)_k$ it is possible to choose $M_{\tilde{\theta}}$ to be independent of $k$.  
\end{proof}

\begin{lemma}
  \label{lem:Okbdd}
  If Assumption~\textup{\ref{ass:Astandass}} is satisfied, then there exists $M_0\geq 1$ such that for all $k$
  \eq{
  \sup_{\gl\in \Omega_k}\; \abs{\gl-\ep}^\ga \norm{R(\gl,A)}\leq M_0.
  }
\end{lemma}

\begin{proof}
  Let $M= \sup_{n\in\N}\norm{A^n}$. From Assumption~\ref{ass:Astandass} we have 
\eq{
\sup_{0<\abs{\varphi-\varphi_k}\leq \eps_A} \abs{\varphi-\varphi_k}^\ga \norm{R(\ep[],A)} \leq M_A.
}
The strong Kreiss resolvent condition implies $(\abs{\gl}-1)\norm{R(\gl,A)}\leq M$ whenever $\abs{\gl}>1$.

Let $\gl = r\ep[]\in \Omega_k$. 
Since $\abs{\varphi-\varphi_k}\leq \eps_A\leq \pi/8$, 
and since $\abs{\varphi-\varphi_k}$ is equal to the arc length between points $\ep[]\in\T$ and $\ep\in\T$, we have 
\ieq{
\abs{e^{i\varphi}-e^{i\varphi_k}}  
\leq \abs{\varphi-\varphi_k}.
}
For $r=1$ the bound 
$\abs{\gl-\ep}^\ga \norm{R(\gl,A)} \leq \abs{\varphi-\varphi_k}^\ga \norm{R(e^{i\varphi},A)}\leq M_A$ follows from~\eqref{eq:Aresolventgrowthorder}. 
On the other hand, if $\varphi=\varphi_k$, $1<r\leq 1+r_A$ and $\gl = re^{i\varphi_k}$, then the strong Kreiss resolvent condition implies
\eq{
\abs{\gl-e^{i\varphi_k}}^\ga \norm{R(\gl,A)}
=(r-1)^\ga \norm{R(\gl,A)}
\leq (r-1) \norm{R(\gl,A)}
\leq M
}
since $(r-1)^\ga \leq r-1$ due to the fact that $\ga\geq 1$ and $0<r-1\leq r_A\leq 1$.
It remains to consider the case $\gl = re^{i\varphi} \in \Omega_k$ with $r>1$ and $\varphi\neq \varphi_k$. 
We can estimate
\eq{
\abs{re^{i\varphi}-e^{i\varphi_k}}
\leq\abs{re^{i\varphi}-e^{i\varphi}} + \abs{e^{i\varphi}-e^{i\varphi_k}}
\leq r-1 + \abs{\varphi-\varphi_k}.
} 
Since $\ga \geq 1 $ and $1<r\leq 2$, we have $(r-1)^\ga \leq r-1$ and
(using the scalar inequality $(a+b)^{\ga} \leq 2^{\ga}(a^{\ga} + b^{\ga})$ for $a,b\geq 0$) we get
\eq{
\abs{\gl-e^{i\varphi_k}}^\ga
\leq  2^{\ga} \left( r-1 + \abs{\varphi-\varphi_k}^\ga \right) 
}
and the resolvent identity $R(re^{i\varphi},A) = R(e^{i\varphi},A) (1- (r-1)e^{i\varphi} R(re^{i\varphi},A))$ implies
\eq{
\MoveEqLeft
\abs{\gl-e^{i\varphi_k}}^\ga \norm{R(\gl,A)} 
\leq 2^{\ga} \left( r-1 + \abs{\varphi-\varphi_k}^\ga \right) \norm{R(\gl,A)} \\
&= 2^{\ga}  (r-1) \norm{R(re^{i\varphi},A)}+ 2^{\ga} \abs{\varphi-\varphi_k}^\ga \norm{R(e^{i\varphi},A)(1 - (r-1) e^{i\varphi} R(re^{i\varphi},A))} \\
&\leq 2^{\ga} M+ 2^{\ga}\abs{\varphi-\varphi_k}^\ga \norm{R(e^{i\varphi},A)} (1+ (r-1) \norm{R(re^{i\varphi},A)} )\\
&\leq 2^{\ga} \left( M +  M_A (1+ M ) \right).
}
Since in each of the situations the bound for $\abs{\gl-e^{i\varphi_k}}^\ga \norm{R(\gl,A)}$ is independent of $k\in \Igw$, the proof is complete.
\end{proof}

\begin{proof}[Proof of Theorem~\textup{\ref{thm:RBAbdd}}]
Let $k\in\Igw$, $\gl\in \Omega_k$, and denote $R_\gl = R(\gl,A)$ and $\gl_k=\gl-\ep$ for brevity.

  We begin by showing that if $\ga = n+\tilde{\ga}$ with $n\in\N$ and $0\leq \tilde{\ga}<1$, then there exists $\tilde{M}\geq 1$ (independent of $k$) such that 
  \eqn{
  \label{eq:ARfracbnd}
  \sup_{\gl\in\Omega_k} \abs{\gl_k}^n \norm{R(\gl,A)\BA_k^{\tilde{\ga}}}\leq \tilde{M}.
  } 
  By Lemma~\ref{lem:Okbdd} there exists $M_0\geq 1$ such that $\abs{\gl-\ep}^\ga \norm{R(\gl,A)}\leq M_0$ for all $k$.  
If $\ga=n$ and $\tilde{\ga}=0$, we have 
\ieq{
\abs{\gl_k}^n \norm{R_\gl\BA_k^{\tilde{\ga}}}
=\abs{\gl_k}^\ga \norm{R_\gl}\leq M_0.
}
Thus the claim is satisfied with $\tilde{M}=M_0$, which is independent of $k$.

If $0<\tilde{\ga}<1$, then 
by Lemma~\ref{lem:Momentineq} there exists a constant $M_{\tilde{\ga}}$ independent of $k$ and $\gl$ such that $\norm{R_\gl\BA_k^{\tilde{\ga}}}\leq M_{\tilde{\ga}} \norm{R_\gl}^{1-\tilde{\ga}}\norm{R_\gl\BA_k}^{\tilde{\ga}}$. Using
  \eqn{
  \label{eq:residentity}
 \ep R_\gl\BA_k
  = R_\gl(\ep- A)
  = 1-\gl_k R_\gl
  }
and the scalar inequality $(a+b)^{\tilde{\ga}} \leq 2^{\tilde{\ga}}(a^{\tilde{\ga}}+b^{\tilde{\ga}})$
we get 
\eq{
\MoveEqLeft \abs{\gl_k}^n  \norm{ R_\gl\BA_k^{\tilde{\ga}}} 
\leq M_{\tilde{\ga}}  \abs{\gl_k}^n \norm{R_\gl}^{1-\tilde{\ga}} \norm{ R_\gl\BA_k}^{\tilde{\ga}}
= M_{\tilde{\ga}}  \abs{\gl_k}^n \norm{R_\gl}^{1-\tilde{\ga}} \norm{1-\gl_k R_\gl}^{\tilde{\ga}}\\
&\leq 2^{\tilde{\ga}}M_{\tilde{\ga}}  \abs{\gl_k}^n \norm{R_\gl}^{1-\tilde{\ga}} (1+\abs{\gl_k}^{\tilde{\ga}}\norm{R_\gl}^{\tilde{\ga}})
\leq 2^{\tilde{\ga}}M_{\tilde{\ga}} \left[ (\abs{\gl_k}^{\frac{n}{1-\tilde{\ga}}} \norm{R_\gl})^{1-\tilde{\ga}} +\abs{\gl_k}^{n+\tilde{\ga}}\norm{R_\gl} \right] .
}
Since $n= \floor{\ga}\geq 1$ we have 
\eq{
\frac{n}{1-\tilde{\ga}}
=\frac{n\ga}{(1-\tilde{\ga})(n+\tilde{\ga})}
=\frac{n\ga}{n-\tilde{\ga}(n-1) -\tilde{\ga}^2}
\geq \ga.
}
Since $\gl\in \Omega_k$, we have $\abs{\gl_k}\leq r_A\leq 1$, and thus $\abs{\gl_k}^{\frac{n}{1-\tilde{\ga}}}\leq \abs{\gl_k}^{\ga}$, and
\eq{
 \abs{\gl_k}^n  \norm{ R_\gl\BA_k^{\tilde{\ga}}} 
\hspace{-.2ex}\leq 2^{\tilde{\ga}}M_{\tilde{\ga}}  \left[ (\abs{\gl_k}^\ga \norm{R_\gl})^{1-\tilde{\ga}} \hspace{-.3ex} +\abs{\gl_k}^\ga\norm{R_\gl} \right]
\leq 2^{\tilde{\ga}}M_{\tilde{\ga}}  \left[ M_0^{1-\tilde{\ga}} \hspace{-.3ex} +M_0 \right]
\leq 2^{\tilde{\ga}+1}M_{\tilde{\ga}} M_0,  
}
since $M_0\geq 1$. Therefore the claim holds with $\tilde{M}=2^{\tilde{\ga}+1}M_{\tilde{\ga}}M_0$, which is independent of $k$.  

We can now show that there exists $M_1\geq 1$ such that~\eqref{eq:ARfracbnd} is satisfied for all $k$. Since $(\BA_k)_k$ is a uniformly sectorial family of operators, by~\cite[Prop. 3.1.1(a)]{Haa06book} there exists $K>0$ such that $\norm{\BA_k^r }\leq K$ for all $0\leq r\leq \ga$ and $k$. Using the identity~\eqref{eq:residentity} repeatedly, we obtain
\eq{
e^{in\varphi_k} R(\gl,A)\BA_k^n
= (-\gl_k)^n R(\gl,A) + \sum_{l=0}^{n-1} (-\gl_k)^{n-1-l}e^{il\varphi_k} \BA_k^l 
}
and thus for $\ga = n+\tilde{\ga}$ (using $\abs{\gl_k}\leq r_A\leq 1$)
\eq{
 \norm{R(\gl,A)\BA_k^\ga}
&= \norm{e^{in\varphi_k} R(\gl,A)\BA_k^n\BA_k^{\tilde{\ga}}}
\leq \abs{\gl_k}^n\norm{ R(\gl,A)\BA_k^{\tilde{\ga}}} + \sum_{l=0}^{n-1} \abs{\gl_k}^{n-1-l} \norm{\BA_k^{l+\tilde{\ga}} }\\
&\leq\tilde{M}  + n  K.
}
Since the bound is independent of both $\gl\in\Omega_k$ and $k$, the proof is complete.
\end{proof}

\begin{lemma}
  \label{lem:RbddOcomp}
Let Assumption~\textup{\ref{ass:Astandass}} be satisfied. There exists $M_2\geq 1$ such that
\eq{
\sup_{\gl\notin \D\cup\bigcup_k \Omega_k  } \norm{R(\gl,A)}\leq M_2.
}
\end{lemma}

\begin{proof}
  Let $\gl =r\ep[]\in \C\setminus \left(\D\cup \bigcup_k \Omega_k \right)$ and let $\gl_0=r_0\ep[]$ be such that $1\leq r_0\leq r$ and $\gl_0$ lies on the boundary of $\D \cup \bigcup_k \Omega_k$. Then either $\gl_0\in \T$, which implies $\norm{R(\gl_0,A)}\leq M_A$ by Assumption~\ref{ass:Astandass}, or otherwise $\gl_0\in \Omega_k$ and $\abs{\gl_0-\ep}=r_A$ for some $k\in \Igw$. By Lemma~\ref{lem:Okbdd} we have that there exists $M_0$ (independent of $k$) such that in the latter case
\eq{
\abs{\gl_0-\ep}^\ga \norm{R(\gl_0,A)}\leq M_0 
\qquad \Leftrightarrow \qquad 
\norm{R(\gl_0,A)}\leq \frac{M_0}{r_A^\ga} .
}
Now, if $M=\sup_{n\in\N}\norm{A^n}$, then $(\abs{\gl}-1) \norm{R(\gl,A)}\leq M$ by the strong Kreiss resolvent condition.
Using the resolvent identity $R(\gl,A) = R(\gl_0,A) + (\gl_0-\gl) R(\gl_0,A)R(\gl,A)$ and $\abs{\gl-\gl_0}=r-r_0\leq r-1 = \abs{\gl}-1$ we have
\eq{
\MoveEqLeft \norm{R(\gl,A)} 
\leq \norm{R(\gl_0,A)}(1 + \abs{\gl-\gl_0} \norm{R(\gl,A)})\\
&\leq \max\set{M_A,M_0/r_A^\ga} (1+(\abs{\gl}-1)\norm{R(\gl,A)})
\leq \max\set{M_A,M_0/r_A^\ga} (1+M).
}
Since the bound is independent of $\gl$, this concludes the proof.
\end{proof}

Combining the above results shows that the growth of the resolvent operator $R(\gl,A)$ near the unit disk $\D$ is cancelled by the operator $\BA_1^\ga\cdots \BA_N^\ga$.

\begin{corollary} 
  \label{cor:RBAbdd}
  If Assumption~\textup{\ref{ass:Astandass}} is satisfied, then 
  \eq{
\sup_{\gl\notin \D \cup \set{\ep}_k}
  \norm{R(\gl,A)\BA_1^\ga\cdots \BA_N^\ga}<\infty.
  }
\end{corollary}

\section{The Preservation of Strong Stability}
\label{sec:strongstabpres}

In this section we present the proof of Theorem~\ref{thm:stabpert}. We begin by studying the change of the spectrum of $A$ under the perturbations.

\begin{theorem}
  \label{thm:specpert}
  Let Assumption~\textup{\ref{ass:Astandass}} be satisfied with $\gb+\gg\geq \ga$. 
  There exists $\gd>0$ such that if
  \eq{
  \norm{\BA_k^{-\gb} B}<\gd, \quad \mbox{and} \quad  \norm{(\BA_k^\ast)^{-\gg}C^\ast}<\gd,
  }
  for every $k$, then
  $\gs(A+BC)\subset \D\cup \set{\ep}_{k=1}^N$ and
  $\set{\ep}_k\subset \gs(A+BC)\setminus \gs_p(A+BC)$.
  In particular, under the above conditions we have
  \eqn{
  \label{eq:Dglunifbdd}
\sup_{\gl\notin \D\cup \set{\ep}_k} \; \norm{(1-CR(\gl,A)B)\inv}<\infty.
  }
\end{theorem}

The proof of Theorem~\ref{thm:specpert} is based on the following two lemmas.

\begin{lemma} 
  \label{lem:CRBbdd}
  Let Assumption~\textup{\ref{ass:Astandass}} be satisfied for some $\gb+\gg\geq \ga$ and let $Y$ be a Banach space.
There exists a constant $M_R\geq 1$ such that if
$B\in \Lin(Y,X)$ and $C\in \Lin(X,Y)$ satisfy $\ran(B)\subset \ran( \BA_k^{\gb})$ and $\ran(C^\ast)\subset \ran( (\BA_k^\ast)^\gg)$ for some $k$, then
\eq{
\norm{CR(\gl,A)B}\leq M_R \norm{\BA_k^{-\gb} B} \norm{(\BA_k^\ast)^{-\gg} C^\ast}
}
for all $\gl\in \Omega_k$.
\end{lemma}

\begin{proof}
  Since $\BA_k^\gb\in \Lin(X)$, the operators $\BA_k^{-\gb}$ and $\BA_k^{-\gb}B$ are closed. Since $\Dom( \BA_k^{-\gb} B)=Y$, the Closed Graph Theorem implies $\BA_k^{-\gb} B\in \Lin(Y,X)$. Similarly $(\BA_k^\ast)^{-\gg} C^\ast \in \Lin(Y,X)$ and $C\BA_k^\gg$ extends to a bounded operator $C_\gg\in \Lin(X,Y)$ with norm $\norm{C_\gg}\leq \norm{(\BA_k^\ast)^\gg C^\ast}$. 
Choose 
$M_k = \norm{(-A)^{\gb+\gg-\ga}} \cdot \sup_{\gl\in \Omega_k}\norm{R(\gl,A)\BA_k^{\ga}}$.
Then for all $\gl\in \Omega_k$ 
\eq{
\norm{CR(\gl,A)B}
&= \norm{C\BA_k^{-\gg} R(\gl,A)\BA_k^{\ga} \BA_k^{\gb+\gg-\ga} \BA_k^{-\gb} B} 
\leq \norm{C_\gg} \norm{R(\gl,A)\BA_k^{\ga}} \norm{ \BA_k^{\gb+\gg-\ga}} \norm{ \BA_k^{-\gb} B} \\
&\leq M_k \norm{\BA_k^{-\gb} B} \norm{(\BA_k)^{-\gg} C^\ast}.
} 
Finally, we can choose $M_R=\max \set{M_1,\ldots,M_N}$.
\end{proof}

\begin{lemma}
  \label{lem:ABCinj}
  Let Assumption~\textup{\ref{ass:Astandass}} be satisfied with $\gb+ \gg \geq \ga$.
  There exists $\gd_0>0$ such that if
  $\norm{\BA_k^{-\gb} B}<\gd_0$  and   $\norm{(\BA_k^\ast)^{-\gg}C^\ast}<\gd_0$
  for all $k$, then
  $\set{\ep}_{k}\subset \gs(A+BC)\setminus\gs_p(A + BC)$.
\end{lemma}

\begin{proof}
  Choose $0\leq\gb_1\leq\gb$ and $0\leq \gg_1\leq\gg$ such that $\gb_1+\gg_1=1$. Let $k\in\Igw$ and assume $\norm{\BA_k^{-\gb_1}B}<1$ and $\norm{(\BA_k^\ast)^{-\gg_1}C^\ast}<1$. The condition $0\leq \gg_1\leq 1$ implies $\ran( \BA_k)\subset \ran( \BA_k^{\gg_1}) \subset X$, and thus $\overline{\Dom( \BA_k^{-\gg_1})} = X$ due to the fact that $\ep\in \gs_c( A)$. The operator $C\BA_k^{-\gg_1}$ has a unique bounded extension $C_{\gg_1}$ with norm $\norm{C_{\gg_1}}=\norm{(\BA_k^\ast)^{-\gg_1}C^\ast}<1$.

  Since $\norm{\ep\BA_k^{-\gb_1}B C_{\gg_1}}\leq\norm{\BA_k^{-\gb_1}B} \norm{ C_{\gg_1}}<1$, the operator $1 - \ep\BA_k^{-\gb_1}BC_{\gg_1}$ is boundedly invertible, and 
  \eq{
  \ep - A - BC
  = \ep \BA_k^{\gb_1} (1 - \epc\BA_k^{-\gb_1}BC_{\gg_1})\BA_k^{\gg_1}.
  } 
  Since $\BA_k^{\gb_1}$ and $\BA_k^{\gg_1}$ are injective and at least one of them is not surjective, the operator $\ep-A-BC$ is injective but not surjective. This implies $\ep \in \gs(A+BC)\setminus\gs_p(A+BC)$.

  Finally, by~\cite[Prop. 3.1.1(a)]{Haa06book} there exists $K>0$ such that $\norm{\BA_k^r }\leq K$ and $\norm{(\BA_k^\ast)^r }\leq K$ for all $0\leq r\leq \gb+\gg$ and $k$. This in particular implies $\norm{\BA_k^{\gb-\gb_1}}\leq K$ and $\norm{(\BA_k^\ast)^{\gg-\gg_1}}\leq K$ for all $k$, and thus $\norm{\BA_k^{-\gb_1}B} \leq K \norm{\BA_k^{-\gb}B}$ and $\norm{(\BA_k^\ast)^{-\gg_1}C^\ast} \leq K\norm{(\BA_k^\ast)^{-\gg}C^\ast}$. This concludes that $\norm{\BA_k^{-\gb_1}B}<1$ and $\norm{(\BA_k^\ast)^{-\gg_1}C^\ast}<1$ can be achieved by choosing a small enough $\gd_0>0$.  \end{proof}

\noindent\textit{Proof of Theorem~\textup{\ref{thm:specpert}}.}
Let $\gb+\gg\geq \ga$.
  By~\cite[Prop. 3.1.1(a)]{Haa06book} there exists 
  $K>0$ such that
$\norm{\BA_k^r }\leq K$ and $\norm{(\BA_k^\ast)^r }\leq K$
for all $0\leq r\leq \gb+\gg$ and $k$. 
We therefore have
  $\norm{B}
 \leq K \norm{\BA_k^{-\gb}B}$
 and
  $\norm{C}
  \leq K\norm{(\BA_k^\ast)^{-\gg}C^\ast}$.  
Lemmas~\ref{lem:RbddOcomp}, \ref{lem:CRBbdd}, and~\ref{lem:ABCinj} now imply that it is possible to choose $\gd>0$ in such a way that if 
  \eq{
\norm{\BA_k^{-\gb} B}<\gd, \quad \mbox{and} \quad  \norm{(\BA_k^\ast)^{-\gg}C^\ast}<\gd,
  }
  for all $k$, then $\norm{CR(\gl,A)B}\leq c <1$ for every $\gl\notin  \D \cup \set{\ep}_k  $, and 
  $\set{\ep}_k\subset \gs(A+BC)\setminus \gs_p(A+BC)$.
The Sherman--Morrison--Woodbury formula
  \eqn{
  \label{eq:SheMor}
  R(\gl,A+BC) = R(\gl,A)+ R(\gl,A) B(1- CR(\gl,A)B)\inv C R(\gl,A)
  }
  now implies that $\gs(A+BC)\subset \D \cup \set{\ep}_k$. Moreover, a standard Neumann series argument shows that 
  $\norm{(1-CR(\gl,A)B)\inv} \leq 1/(1-c)$ for every $\gl\notin \D\cup \set{\ep}_k$, which in turn concludes that~\eqref{eq:Dglunifbdd} is satisfied. 
\hfill\qed

The following theorem characterizes power boundedness of a discrete semigroup on a Hilbert space~\cite{Eis10book}.

\begin{theorem}
  \label{thm:unifbddconds}
  Let $A\in\Lin(X)$ on a Hilbert space $X$ be such that $\gs(A)\subset \overline{\D}$. The discrete semigroup $\Asg$ is power bounded if and only if 
      for all $x,y\in X$ 
      \eq{
      \sup_{1<r\leq 2}\, (r-1) \int_0^{2\pi} \left( \norm{R(r\ep[],A)x}^2+ \norm{R(r\ep[],A)^\ast y}^2 \right) d\varphi <\infty.
      } 
\end{theorem}

\begin{lemma}
  \label{lem:BCfinrankint}
  If $\tilde{B}\in \Lin(\C^p,X)$, then 
  \eq{
  &\sup_{1<r\leq 2} \; (r-1)\int_0^{2\pi} \norm{R(r\ep[],A)\tilde{B}}^2 d\varphi <\infty,
  \quad
  &\sup_{1<r\leq 2} \; (r-1)\int_0^{2\pi}  \norm{ R(r\ep[],A)^\ast \tilde{B}}^2 d\varphi<\infty.
  } 
\end{lemma}

\begin{proof}
  The claim follows directly from the fact that there exist $\set{b_j}_{j=1}^p\subset X$ such that $\tilde{B}u= \sum_{j=1}^p u_jb_j $ for $u\in \C^p$, and for any $R\in \Lin(X)$ we have
  $\norm{R \tilde{B}}^2 \leq  \sum_{j=1}^p\, \norm{R b_j}^2$ .
\end{proof}

\begin{lemma} 
  \label{lem:RBCRest}
Let Assumption~\textup{\ref{ass:Astandass}} be satisfied for some $\gb+\gg\geq \ga$ and let $k\in\Igw$.
  There exists a function $f_k: \C\setminus \left( \D\cup \set{\ep[l]}_{l=1}^N \right)\rightarrow \R^+$ such that
  \eq{
  \norm{R(\gl,A)B}\norm{CR(\gl,A)}\leq f_k(\gl) \qquad \forall \gl\in\Omega_k,
  }
   and $f_k(\cdot)$ has the properties $\sup_{0<\abs{\varphi-\varphi_k}\leq \eps_A}\abs{\varphi-\varphi_k}^\ga f_k(\ep[])<\infty$ and
  \eqn{
  \label{eq:fkint}
  \sup_{1<r\leq 2} \; (r-1)\int_0^{2\pi} f_k(r\ep[])^2 d\varphi<\infty.
  }
\end{lemma}

\begin{proof}
Choose $0\leq \gb_1\leq \gb$ and $0\leq \gg_1\leq \gg$ such that $\gb_1+\gg_1=\ga$.
  For brevity, denote $R_\gl = R(\gl,A)$ and $\gl_k=\gl-\ep$. 
  Moreover, denote $B_{\gb_1} = \BA_k^{-\gb_1} B\in \Lin(\Y,X)$ and $\tilde{C}_{\gg_1}= (\BA_k^\ast)^{-\gg_1}C^\ast\in \Lin(\Y,X)$. 
  
Let $M_1\geq 1$ be as in Theorem~\ref{thm:RBAbdd}. By Lemma~\ref{lem:Momentineq} there exist constants $M_{\gb_1},M_{\gg_1}\geq 1$ such that 
 for every $\gl\in \Omega_k$ we have
  \eq{
  \norm{R_\gl B  }
  &=\norm{\BA_k^{\gb_1} R_\gl   B_{\gb_1} }
  \leq M_{\gb_1} \norm{R_\gl   B_{\gb_1} }^{1-\gb_1/\ga} \norm{ \BA_k^\ga R_\gl   B_{\gb_1} }^{\gb_1/\ga}\\
  &\leq M_{\gb_1} \norm{R_\gl   B_{\gb_1} }^{1-\gb_1/\ga} \norm{\BA_k^\ga  R_\gl}^{\gb_1/\ga} \norm{   B_{\gb_1} }^{\gb_1/\ga}\\
  \norm{R_\gl^\ast C^\ast  }
  &=\norm{(\BA_k^\ast)^{\gg_1} R_\gl^\ast   \tilde{C}_{\gg_1} }
  \leq M_{\gg_1} \norm{ R_\gl^\ast   \tilde{C}_{\gg_1} }^{1-\gg_1/\ga} \norm{ (\BA_k^\ast)^\ga R_\gl^\ast   \tilde{C}_{\gg_1} }^{\gg_1/\ga}\\
  &\leq M_{\gg_1} \norm{ R_\gl^\ast   \tilde{C}_{\gg_1} }^{1-\gg_1/\ga} \norm{R_\gl\BA_k^\ga}^{\gg_1/\ga} \norm{\tilde{C}_{\gg_1} }^{\gg_1/\ga}.
  }
  Therefore 
  for $K = M_{\gb_1} M_{\gg_1} M_1\norm{B_{\gb_1} }^{\gb_1/\ga} \norm{\tilde{C}_{\gg_1} }^{\gg_1/\ga}$ we have
  \eq{
  \norm{R_\gl B} \norm{CR_\gl} 
  & \leq
  K \norm{R_\gl B_{\gb_1}}^{1-\gb_1/\ga} \norm{R_\gl^\ast \tilde{C}_{\gg_1}}^{1-\gg_1/\ga}.
  } 
  Define $f_k(\cdot)$ by $f_k(\gl) = K \norm{R_\gl B_{\gb_1}}^{1-\gb_1/\ga} \norm{R_\gl^\ast \tilde{C}_{\gg_1}}^{1-\gg_1/\ga}$ for all $\gl\in \C\setminus \left(\D \cup \set{\ep[l]}_{l=1}^N \right)$. We will now show that $f_k(\cdot)$ has the desired properties. 

  Since  $1-\gb_1/\ga + 1-\gg_1/\ga=1$, for all $\varphi\in[0,2\pi]$ with $0<\abs{\varphi-\varphi_k}\leq \eps_A$ we have from Assumption~\ref{ass:Astandass} that
  \eq{
 \abs{\varphi-\varphi_k}^\ga f_k(\ep[]) 
  &\leq \abs{\varphi-\varphi_k}^\ga  \norm{R(\ep[],A)} K \norm{B_{\gb_1}}^{1-\gb_1/\ga} \norm{ \tilde{C}_{\gg_1}}^{1-\gg_1/\ga}\\
&\leq  M_A K \norm{B_{\gb_1}}^{1-\gb_1/\ga} \norm{ \tilde{C}_{\gg_1}}^{1-\gg_1/\ga}.
  }
  This concludes that $\sup_{0<\abs{\varphi-\varphi_k}\leq \eps_A}\abs{\varphi-\varphi_k}^\ga f_k(\ep[])<\infty$.

  Moreover, if we denote $q=1/(1-\gb_1/\ga)$, $q' = 1/(1-\gg_1/\ga)$, then $1/q+1/q'=1$ and the H\"older inequality implies
      \eq{
      \int_0^{2\pi} f_k(r\ep[])^2 d\varphi
&=K^2 \int_0^{2\pi}\norm{ R(r\ep[],A) B_{\gb_1}}^{\frac{2}{q}} \norm{ R(r\ep[],A)^\ast \tilde{C}_{\gg_1}}^{\frac{2}{q'}} d\varphi\\
      &\leq K^2 \left(\int_0^{2\pi} \norm{ R(r\ep[],A) B_{\gb_1}}^2 d\varphi \right)^\frac{1}{q}
      \left(\int_0^{2\pi} \norm{ R(r\ep[],A)^\ast \tilde{C}_{\gg_1}}^2 d\eta \right)^\frac{1}{q'}
      }
      which immediately implies~\eqref{eq:fkint} by Lemma~\ref{lem:BCfinrankint}.
\end{proof}

\noindent\textit{Proof of Theorem~\textup{\ref{thm:stabpert}}.}
  Let $\gd>0$ be chosen as in Theorem~\ref{thm:specpert} and assume
  $\norm{\BA_k^{-\gb}B}<\gd$, and $\norm{(\BA_k^\ast)^{-\gg}C^\ast}<\gd$ for all $k$. 
  By Theorem~\ref{thm:specpert} there exists $M_D\geq 1$ such that we have $\norm{(1-CR(\gl,A)B)\inv}\leq M_D$ for all $\gl\notin \D\cup \set{\ep}_{k=1}^N$. 
  We begin the proof by showing that the semigroup $\Asg[(A+BC)]$ is power bounded.
  
Let $x\in X$ and
for brevity denote $R_\gl = R(r\ep[],A)$ and $D_\gl = 1-CR(r\ep[],A)B$. 
Using the Sherman--Morrison--Woodbury formula~\eqref{eq:SheMor} and the scalar inequality $(a+b)^2\leq 2 (a^2+b^2)$ for $a,b\geq 0$
we can estimate
\eq{
\MoveEqLeft \int_0^{2\pi} \norm{R(r\ep[],A+BC)x}^2 d\varphi
= \int_0^{2\pi} \norm{R_\gl x + R_\gl BD_\gl\inv CR_\gl x}^2 d\varphi\\
&\leq 2\int_0^{2\pi} \norm{R_\gl x}^2 d\varphi + 2M_D^2 \norm{x}^2 \int_0^{2\pi}
 \norm{R_\gl B}^2  \norm{ CR_\gl}^2 d\varphi .
}
Similarly, using $\norm{(R_\gl BD_\gl\inv CR_\gl)^\ast} = \norm{R_\gl BD_\gl\inv CR_\gl}\leq M_D\norm{R_\gl B}\norm{CR_\gl}$ we get
\eq{
\MoveEqLeft \int_0^{2\pi} \norm{R(r\ep[],A+BC)^\ast x}^2 d\varphi
= \int_0^{2\pi}
 \norm{R_\gl^\ast x + (R_\gl BD_\gl\inv CR_\gl)^\ast x}^2 d\varphi\\
&\leq 2\int_0^{2\pi}
 \norm{R_\gl^\ast x}^2 d\varphi + 2M_D^2 \norm{x}^2 \int_0^{2\pi} \norm{R_\gl B}^2  \norm{ CR_\gl}^2 d\varphi .
}
The above estimates together with 
 Theorem~\ref{thm:unifbddconds} imply that the semigroup generated by $A+BC$ is uniformly bounded if 
\eqn{
\label{eq:stabpertRBCRint}
\sup_{1<r\leq 2} \; (r-1)\int_0^{2\pi} \norm{R_\gl B}^2  \norm{ CR_\gl}^2 d\varphi 
<\infty.
}

For all $k\in\Igw$ let $f_k(\cdot)$ be the functions in Lemma~\ref{lem:RBCRest}. By Lemma~\ref{lem:RbddOcomp} we can choose $M_2\geq 1$ such that $\norm{R(\gl,A)}\leq M_2$ for all $\gl\notin \D\cup \bigcup_k \Omega_k $.
Let $1<r\leq 2$. For each $k\in\Igw$ denote by $E_k^r\subset [0,2\pi]$ the interval such that $r\ep[] \in\Omega_k$ if and only if $\varphi\in E_k^r$. Finally, denote $E^r = [0,2\pi]\setminus \left( \bigcup_k E_k^r \right)$. Now
\eq{
\int_0^{2\pi} \norm{R_\gl B}^2  \norm{ CR_\gl}^2 d\varphi 
&=  \int_{E^r} \norm{R_\gl B}^2  \norm{ CR_\gl}^2 d\varphi 
+ \sum_{k=1}^N  \int_{E_k^r} \norm{R_\gl B}^2  \norm{ CR_\gl}^2 d\varphi \\
&\leq \int_{E^r}  M_2^2\norm{ B}^2 \norm{ C}^2  M_2^2d\varphi 
+ \sum_{k=1}^N  \int_{E_k^r} f_k(r\ep[])^2 d\varphi \\
&\leq 2\pi M_2^4 \norm{B}^2\norm{C}^2   
+ \sum_{k=1}^N  \int_0^{2\pi} f_k(r\ep[])^2 d\varphi ,
}
which immediately implies~\eqref{eq:stabpertRBCRint} 
by Lemmas~\ref{lem:BCfinrankint}~and~\ref{lem:RBCRest}, and therefore the semigroup $\Asg[(A+BC)]$ is power bounded.

Since the perturbed semigroup is power bounded and $X$ is a Hilbert space, 
Theorem 2.9 and Corollary 2.11 in~\cite{Eis10book}
imply
that $\gs(A+BC)\cap \T\subset \gs_p(A+BC)\cup \gs_c(A+BC)$. However, by Theorem~\ref{thm:specpert} we have that
$\set{\ep}_k\subset\gs(A+BC)\setminus\gs_p(A+BC) $.
Together these properties conclude that $\ep\in \gs_c(A+BC)$ for all $k$.

Theorem~\ref{thm:specpert} shows that $\gs(A+BC)\cap \T = \set{\ep}_{k=1}^N$ is finite and $\gs_p(A+BC)\cap \T = \varnothing$. The discrete Arent--Batty--Lyubich--V\~{u} Theorem~\cite[Thm. 2.18]{Eis10book} therefore concludes that the semigroup $\Asg[(A+BC)]$ is strongly stable.

It remains to show that for all $k$ we have 
\eqn{
\label{eq:stabpertgrowthest}
\sup_{0<\abs{\varphi-\varphi_k}\leq \eps_A} \abs{\varphi-\varphi_k}^\ga \norm{R(\ep[],A+BC)}<\infty.
} 
Let $k$ be arbitary. By Lemma~\ref{lem:RBCRest} there exists $M_k\geq 1$ such that $\abs{\varphi-\varphi_k}^\ga f_k(\ep[])\leq M_k$ whenever $0<\abs{\varphi-\varphi_k}\leq\eps_A$. 
The Sherman--Morrison--Woodbury formula~\eqref{eq:SheMor} implies that for all $\varphi\in[0,2\pi]$ satisfying $0<\abs{\varphi-\varphi_k}\leq \eps_A$ we have
\eq{
\MoveEqLeft\norm{R(\ep[],A+BC)}
\leq\norm{R(\ep[],A)}  + \norm{R(\ep[],A) B} M_D \norm{CR(\ep[],A)} \\
&\leq\norm{R(\ep[],A)}  + M_D f_k(\ep[] ),
}
and thus
\eq{
\abs{\varphi-\varphi_k}^\ga \norm{R(\ep[],A+BC)}
&\leq\abs{\varphi-\varphi_k}^\ga\norm{R(\ep[],A)}  + M_D \abs{\varphi-\varphi_k}^\ga  f_k(\ep[] )\\
&\leq M_A + M_D M_k.
}
This concludes that~\eqref{eq:stabpertgrowthest} is satisfied. 
On the other hand, if $\abs{\varphi-\varphi_k}>\eps_A$ for all $k$, then a similar estimate yields
\ieq{
\norm{R(\ep[],A+BC)}
\leq M_A+M_D \norm{B} \norm{C} M_A^2.
}
This concludes the proof.  
\hfill\qed

\end{document}